\documentclass[11pt]{amsart}
\usepackage{geometry}
\usepackage{amssymb, amsmath, amsthm, tikz,enumerate,mathrsfs}
\usepackage{pdfpages}
\usepackage{hyperref}
\usepackage{textcomp}
\usepackage{listings}

\usepackage{fullpage}

\newtheorem{lemma}{Lemma}[section]
\newtheorem{proposition}[lemma]{Proposition}
\newtheorem{theorem}[lemma]{Theorem}

\newtheorem{conjecture}[lemma]{Conjecture}

\newtheorem{question}[lemma]{Question}

\theoremstyle{definition}
\newtheorem{definition}[lemma]{Definition}

\makeatletter
\newcommand{\setword}[2]{%
  \phantomsection
  #1\def\@currentlabel{\unexpanded{#1}}\label{#2}%
}
\makeatother

\numberwithin{equation}{section}

\newcommand{\Ball}{\mathbb{B}}

\newcommand{\Rrr}{\mathbb{R}}

\newcommand{\Hf}{\mathcal{H}}
\newcommand{\Tf}{\mathscr{T}}

\newcommand{\Tt}{\mathcal{T}}

\DeclareMathOperator{\cross}{cross}

\DeclareMathOperator{\Span}{Span}
\DeclareMathOperator{\Vol}{Vol}

\makeatletter
\@namedef{subjclassname@2020}{%
  \textup{2020} Mathematics Subject Classification}
\makeatother

\begin{document}

\title{Conjectures for cutting pizza with Coxeter arrangements}

\author{Richard Ehrenborg}

\address{Department of Mathematics, University of Kentucky, Lexington,
  KY 40506-0027, USA.\hfill\break \tt http://www.math.uky.edu/\~{}jrge/,
  richard.ehrenborg@uky.edu.}



\date{\today.}

\begin{abstract}
We are interested in conjecturing the sign of the pizza quantity
$P(\Hf,\mathbb{B}(a,R))$ for the irreducible Coxeter arrangements $\Hf$ of
type $A_{n}$, where $n \equiv 2$ or $3 \bmod 4$, and
type $D_{n}$, where $n$ is odd.
Our approach is to express the pizza quantity in terms of
the pizza quantity of subarrangements known as $2$-structures,
and we obtain the first non-zero term in the multivariate Taylor expansion.
\end{abstract}

\maketitle

\section{Introduction}

For an hyperplane arrangement $\Hf$ in an inner product space $V$
let $\Tf$ denote the set of chambers of the arrangement~$\Hf$.
Let $T_{0}$ be a base chamber
and let the sign $(-1)^{T}$ be $-1$ to the number of hyperplanes that separate
the chamber $T$ from the base chamber $T_{0}$.
The {\em pizza quantity} 
for a measurable set $K$ in the space $V$
is defined in~\cite{EMR_pizza}
to be the alternating sum
\begin{align}
P(\Hf,K)
& = 
\sum_{T \in \Tf} (-1)^{T} \cdot \Vol(K \cap T) .
\label{equation_pizza_definition}
\end{align}
The notion to have in mind is that two friends are sharing a pizza.
They cut the pizza with the hyperplanes in $\Hf$ and then alternate
taking slices. The pizza quantity measures the difference in the amount
of pizza they each receive.

Recall that a Coxeter arrangement is an hyperplane arrangement which is preserved under
reflections in each of its hyperplanes, and the group generated is finite.
For balls and for Coxeter arrangements satisfying the parity condition,
the following result states that the two friends can share the pizza equally
as long as the ball contains the origin.
\begin{theorem}[Ehrenborg--Morel--Readdy~\cite{EMR_pizza}]
Let $\Hf$ be a Coxeter arrangement 
on a finite-dimensional
inner product space $V$ such that $|\Hf| \geq \dim(V)$.
Assume that the ball $\mathbb{B}(a,R)$ contains the origin,
that is, $R \geq \|a\|$.
\begin{itemize}
\item[(i)]
If the number of hyperplanes is strictly greater than the dimension of~$V$
and has the same parity as that dimension
then the pizza quantity $P(\Hf,\mathbb{B}(a,R))$ vanishes.
\item[(ii)]
If the number of hyperplanes is equal to the dimension of $V$
then
$P(\Hf,\mathbb{B}(a,R))$ is independent of the radius~$R$.
\end{itemize}
\label{theorem_second_in_introduction}
\end{theorem}
This result was known in a few cases:
For the type $B_{2}$ it was stated as a problem by Upton~\cite{Upton}.
For the dihedral arrangement $I_{2}(2k)$, where $k \geq 2$,
this result was first proved by Goldberg~\cite{Gol}.
For a product arrangement of type $A_{1} \times I_{2}(2k)$, again $k \geq 2$,
it was proved by Frederickson~\cite{Fred}.
Furthermore, Brailov had proved the result independently
for an arrangement of type $B_{n}$ in~\cite{Brailov}.

Note that Theorem~\ref{theorem_second_in_introduction} does not
yield a satisfactory answer for the arrangements
where the parities of the number of hyperplanes and the dimension differ.
The irreducible types of Coxeter arrangements that do not satisfy the parity condition are
$I_{2}(2k+1)$ where $k \geq 1$,
$A_{n}$ where $n \equiv 2$ or $3 \bmod 4$, and
$D_{n}$ where $n$ is odd.
For the two-dimensional dihedral arrangements
Mabry and Deiermann~\cite{MaDe}
obtained the following result.
\begin{theorem}[Mabry--Deiermann~\cite{MaDe}]
Let $\Hf$ be an arrangement of type $I_{2}(m)$ where $m \geq 3$ and $m$ is odd.
Let $a$ be a point in the interior of a chamber $T$ of the arrangement
such that $\|a\| \leq R$.
Then the sign of the pizza quantity is given by
\begin{align*}
(-1)^{(m-3)/2} \cdot (-1)^{T} \cdot P(\Hf,\mathbb{B}(a,R)) & > 0 .
\end{align*}
\end{theorem}
In terms of pizza, assume that you and your friend are sharing a two-dimensional pizza
and you would like to maximize the amount of pizza you receive. If you cut the pizza with
$m \equiv 3 \bmod 4$ lines in a dihedral arrangement
you should make sure you take the slice containing the center of the pizza.
On the other hand, if
$m \equiv 1 \bmod 4$ you should avoid the slice containing the center.

Hence it is natural to conjecture the following:

\begin{conjecture}[Ehrenborg--Morel--Readdy~\cite{EMR_pizza}]
Let $\Hf$ be a Coxeter arrangement and let $a\in V$ such that $\|a\|\leq R$.
Assume that the number of hyperplanes of $\Hf$ is greater than the dimension of~$V$
and that these two quantities differ in their parities.
Then the pizza quantity~$P(\Hf,\Ball(a,R))$ 
is zero if and only if the center~$a$ lies
on one of the hyperplanes of~$\Hf$.
\label{conjecture_conjecture}
\end{conjecture}
This conjecture is basically
Conjecture~9.5 in~\cite{EMR_pizza}.
Our purpose in this paper is to generate data such that
we can predict the sign of the pizza quantity given the sign of the
open chamber that contains the center $a$, in the case when
the Coxeter arrangement has type $A_{n}$ or $D_{n}$.
Our approach for obtaining data follows the idea
of~\cite{EMR_pizza_2} of expressing the pizza quantity
in terms of pizza quantities of smaller subarrangements,
known as $2$-structures.
These subarrangements were introduced by Herb
in the study of the characters of discrete series representations~\cite{Herb-2S}.
In the cases we are interested in, the $2$-structures all have
the type $A_{1}^{k}$. For the type $A_{1}^{k}$ we obtain the multivariate Taylor series
for the pizza quantity, yielding the Taylor series for the desired
pizza quantity of types~$A_{n}$ and~$D_{n}$.
We were able to do the (computer) calculations up to and including
dimension~$7$, before the polynomials became too unwieldy.
Lastly, we return to type $D_{n}$, where we make a conjecture
that would imply the sign for arrangements of type $D_{n}$
where $n$~is odd.

\section{Coxeter arrangements, pseudo-root systems and $2$-structures}

Given two hyperplane arrangements $\Hf_{1}$ and $\Hf_{2}$ in the two spaces $V_{1}$
and $V_{2}$, respectively.
The product arrangement $\Hf_{1} \times \Hf_{2}$ in $V_{1} \times V_{2}$ is defined to be
\begin{align*}
\Hf_{1} \times \Hf_{2}
& =
\{H \times V_{2} : H \in \Hf_{1}\} \cup \{V_{1} \times H : H \in \Hf_{2}\} .
\end{align*}
A {\em Coxeter arrangement} $\Hf$ is a finite set of hyperplanes
such that the arrangement is invariant under the orthogonal reflections
in any of its hyperplanes, and the group (known as the Coxeter group) 
generated by these reflections is finite.
Note that the class of Coxeter arrangements is closed under
Cartesian products. Furthermore, call an arrangement
{\em irreducible} if it cannot be written as a Cartesian product.
The irreducible Coxeter arrangements have been classified;
see for instance~\cite{Bourbaki}.
Their types are
$A_{n}$ for $n \geq 1$;
$B_{n}$ for $n \geq 2$ (also known as type $B/C$);
$D_{n}$ for $n \geq 3$ (note that $D_{3}$ is isomorphic to $A_{3}$);
$E_{6}$, $E_{7}$ and $E_{8}$;
$F_{4}$;
$H_{3}$ and $H_{4}$;
and
the dihedral arrangement~$I_{2}(m)$ where $m \geq 2$
(observe that $I_{2}(2) \cong A_{1}^{2}$, $I_{2}(3) \cong A_{2}$ and $I_{2}(4) \cong B_{2}$).

Let $V$ be a finite-dimensional real vector space with an inner product
$(\cdot,\cdot)$.
For every non-zero vector $\alpha\in V$,
let $H_{\alpha} = \{x \in V : (\alpha,x) = 0\}$
denote the hyperplane orthogonal to $\alpha$,
and let $s_{\alpha}$ be the orthogonal reflection in 
the hyperplane $H_{\alpha}$.
\begin{definition}
A subset $\Phi$ of $V$ is a
{\em pseudo-root system} if:
\begin{itemize}
\item[(a)] $\Phi$ is a finite set of unit vectors;
\item[(b)] for all $\alpha,\beta\in\Phi$, we have $s_\beta(\alpha)\in\Phi$.
\end{itemize}
\end{definition}
Note that a consequence of condition~(b) is that $\alpha \in \Phi$ implies $-\alpha \in \Phi$
by setting $\alpha = \beta$.
Elements of $\Phi$ are called \emph{pseudo-roots}.
Let the {\em Coxeter group} $W$ be the group generated by all the reflections
$s_{\alpha}$ for $\alpha$ in the pseudo-root system $\Phi$.
For a group element~$w$ in $W$ the sign of the element $w$
is the sign of the chamber $w(T_{0})$, that is, $(-1)^{w} = (-1)^{w(T_{0})}$.
Note that this definition is same as defining the sign of $w$
as $(-1)^{\ell(w)}$ where $\ell$ is the length function in the Coxeter system.

The elements of a pseudo-root system $\Phi$ are naturally divided
into positive and negative roots,
that is, the following disjoint union holds: $\Phi = \Phi^{+} \sqcup \Phi^{-}$.
We use the characterization given in~\cite[Lemma~B.1.9]{EMR}.
Define the {\em lexicographic order} on $\Rrr^{r}$ as follows:
$(x_{1}, \ldots, x_{r}) <_{\text{lex}} (y_{1}, \ldots, y_{r})$
if there is an index $i$ such that
$x_{1} = y_{1}, \ldots, x_{i-1} = y_{i-1}$ and $x_{i} < y_{i}$.
Then Lemma~B.1.9 is the following statement.
\begin{lemma}
Let $\Phi$ be a pseudo-root system in~$V$
and assume that $v_{1}$ through~$v_{r}$ are linearly independent elements of~$V$
such that no pseudo-root of~$\Phi$ is orthogonal to every~$v_{i}$.
Then the set~$\Phi^{+}$ defined by
\begin{align*}
\Phi^{+}
& =
\{\alpha \in \Phi : ((\alpha,v_{1}), \ldots, (\alpha,v_{r})) >_{\text{lex}} (0, \ldots, 0) \}
\end{align*}
is a system of positive pseudo-roots.
\end{lemma}

Given a pseudo-root system $\Phi$.
We obtain a Coxeter arrangement by considering
the collection $\Hf_{\Phi} = \{H_{\alpha} : \alpha \in \Phi^{+}\}$.
Moreover, the base chamber $T_{0}$ is
on the positive side on each hyperplane, that is,
$T_{0} = \bigcap_{\alpha \in \Phi^{+}} \{x \in V : (\alpha,x) \geq 0\}$.
Similarly, given a Coxeter arrangement
we obtain a pseudo-root system by considering the unit
vectors orthogonal to the hyperplanes.
The positive pseudo-roots are on the same side as the base chamber,
that is,
$\Phi^{+} = \{\alpha \in \Phi :  \forall x \in T_{0} \:\: (\alpha,x) \geq 0\}$.

For a root system $\Phi$ with positive roots $\Phi^{+}$
we define the Jacobian to be the product
\begin{align*}
J(a)
& =
\prod_{\alpha \in \Phi^{+}} (\alpha,a) .
\end{align*}
Note that the Jacobian is the polynomial which is zero on
all the hyperplanes in the arrangement associated with $\Phi$.
Moreover, any skew symmetric polynomial on $V$ factors
as the Jacobian times a polynomial invariant under $W$.

The two types of
irreducible Coxeter arrangements, equivalently pseudo-root systems,
we consider in this paper are the types $A_{n}$ and $D_{n}$.
For each type we present the positive pseudo-roots and the Jacobian.
\begin{itemize}
\item[$(A_{n})$]
Let $V$ be the $n$-dimensional space
$\{(x_{1}, \ldots, x_{n+1}) \in \Rrr^{n+1} : x_{1} + \cdots + x_{n+1} = 0\}$.
The positive pseudo-roots are
\begin{align*}
\Phi^{+}
& =
\{ (e_{i} - e_{j})/\sqrt{2} : 1 \leq i < j \leq n+1\} .
\end{align*}
In this type the Jacobian
is the Vandermonde product up to a scalar
\begin{align*}
J(a) & =
{2}^{- (n+1) n /4} \cdot 
\prod_{1 \leq i < j \leq n+1} (a_{i} - a_{j}) .
\end{align*}

\item[$(D_{n})$]
Let $V$ be $\Rrr^{n}$.
The positive pseudo-roots are
\begin{align*}
\Phi^{+}
& =
\{ (e_{i} \pm e_{j})/\sqrt{2} : 1 \leq i < j \leq n\} .
\end{align*}
The Jacobian in this type is given by
\begin{align*}
J(a) & =
{2}^{-n \cdot (n-1)/2} \cdot 
\prod_{1 \leq i < j \leq n} (a_{i} - a_{j}) \cdot (a_{i} + a_{j}) \\
& =
{2}^{-n \cdot (n-1)/2} \cdot 
\prod_{1 \leq i < j \leq n} (a_{i}^{2} - a_{j}^{2}) ,
\end{align*}
that is, the Vandermonde product in the
squares $a_{1}^{2}, \ldots, a_{n}^{2}$.
\end{itemize}

We now present $2$-structures and their basic properties.
Herb introduced them for crystaligraphic root systems
in order to study the characters of discrete series representations~\cite{Herb-2S}.
We use Definition~B.2.1 of~\cite{EMR},
which works for pseudo-root systems.
\begin{definition}
Let $\Phi$ be a pseudo-root system with Coxeter group $W$.
A \emph{$2$-structure} for $\Phi$ is a subset~$\varphi$ of~$\Phi$
satisfying the following properties:
\begin{itemize}
\item[(a)] 
The subset $\varphi$ is a disjoint union
$\varphi=\varphi_1\sqcup\varphi_2\sqcup\cdots\sqcup\varphi_r$,
where the $\varphi_i$ are pairwise orthogonal subsets
of $\varphi$ and each of them is an irreducible pseudo-root system of type
$A_1$, $B_2$ or $I_2(2^k)$ for $k\geq 3$.
\item[(b)] Let $\varphi^+=\varphi\cap\Phi^+$. If $w$ is an element in $W$ such that
$w(\varphi^+)=\varphi^+$ then the sign of $w$ is positive, that is, $(-1)^{w} = 1$.
\end{itemize}
\label{def_2_structure}
We denote by $\Tt(\Phi)$ the set of $2$-structures for $\Phi$.
\end{definition}

Furthermore, to each $2$-structure $\varphi$, we can associate a 
sign $\epsilon(\varphi)$; see~\cite[Section~3 and Appendix~B]{EMR} .
Two essential properties of the sign are described as follows.
\begin{lemma}
\begin{itemize}
\item[(a)]
Let $w$ be in the Coxeter group such that
$w(\varphi^{+}) \subseteq \Phi^{+}$. Then the following identity holds:
$\epsilon(w(\varphi)) = (-1)^{w} \cdot \epsilon(\varphi)$.
\item[(b)]
The alternating sum of all the signs of $2$-structures is equal to $1$, that is,
\begin{align*}
\sum_{\varphi \in \Tt(\Phi)} \epsilon(\varphi) & = 1 .
\end{align*}
\end{itemize}
\label{lemma_sign_for_2-structures}
\end{lemma}
For part~(a) see~Lemma~B.2.10 in~\cite{EMR} and for part~(b)
Proposition~3.1.1 in the same reference.

Next we need Theorem~3.2.1 in~\cite{EMR}. However, we only need
it in the case when the valuation is the volume. Hence we can consider the
more straightforward case: when the valuation vanishes on lower
dimensional cones.
Let $\mathcal{C}(V)$ be the collection of cones in the vector space $V$
and let $\mathcal{B}(V)$ be the smallest Boolean algebra that contains
$\mathcal{C}(V)$ and is closed under unions and intersections.
A valuation $v$ on $\mathcal{B}(V)$ is a function such that
$v(\emptyset) = 0$ and $v(S) + v(T) = v(S \cap T) + v(S \cup T)$
for all $S, T \in \mathcal{B}(V)$.
For a hyperplane arrangement $\Hf$ and a valuation $v$ define
the pizza quantity with respect to the valuation $v$ as
the alternating sum
\begin{align*}
\Pi(\Hf, v)
& =
\sum_{T \in \Tf} (-1)^{T} \cdot v(T) .
\end{align*}
Theorem~3.2.1 in~\cite{EMR} 
implies the following statement.
\begin{theorem}
Let $\Hf$ be a Coxeter arrangement with pseudo-root system~$\Phi$.
Then the pizza quantity with respect to the valuation $v$
can be expressed as
\begin{align*}
\Pi(\Hf, v)
& =
\sum_{\varphi \in \Tt(\Phi)} \epsilon(\varphi) \cdot \Pi(\Hf_{\varphi}, v) ,
\end{align*}
where $\varphi$ ranges over all $2$-structures of the 
pseudo-root system~$\Phi$
and $\Hf_{\varphi}$ denotes the hyperplane arrangement associated
to the pseudo-root system $\varphi$.
\label{theorem_2-structure}
\end{theorem}

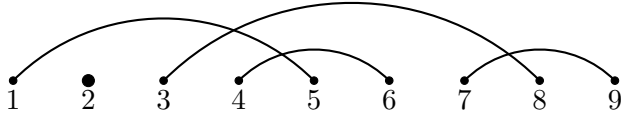
\begin{figure}
\begin{center}
\begin{tikzpicture}
\foreach \x in {1,2, ..., 9}{\draw[fill] (\x,0) circle (0.5 mm) node[below] {$\x$};};
\draw[fill] (2,0) circle (0.8 mm);
\draw[thick] (1,0) to [bend left=45] (5,0);
\draw[thick] (3,0) to [bend left=45] (8,0);
\draw[thick] (4,0) to [bend left=45] (6,0);
\draw[thick] (7,0) to [bend left=45] (9,0);
\end{tikzpicture}
\end{center}
\caption{A maximal matching on the set $\{1,2, \ldots, 9\}$.
Note that the number of crossings is $3$ and the isolated vertex is at $2$.
Hence the sign of this matching is $(-1)^{2-1} \cdot (-1)^{3} = +1$.}
\label{figure_example_of_a_matching}
\end{figure}

For a pseudo-root system of type $A_{n}$ or of type $D_{n}$,
each 2-structure's index set is a set of maximal matchings
on a set $\{1,2, \ldots, r\}$.
Let the matching $M$ be given by
\begin{align}
M
& =
\{(i_{1}<j_{1}), (i_{2}<j_{2}), \ldots, (i_{m}<j_{m})\} ,
\label{equation_matching}
\end{align}
where $m = \lfloor r/2 \rfloor$.
Furthermore, the sign of the matching $M$ is given by the following rule;
see~\cite[Section~6]{Ehrenborg_Morel_Readdy}.
A {\em crossing} of a matching is two pairs of edges
$\{a<b\}, \{c<d\} \in M$ such that $a < c < b < d$.
Let $\cross(M)$ denote the number of crossings of the matching.
If $r$ is odd, there is an isolated vertex, say $p$.
Then the sign of $M$ is given by
\begin{align*}
(-1)^{M}
& =
\begin{cases}
(-1)^{\cross(M)} & \text{ if $r$ is even,} \\
(-1)^{\cross(M)+p-1} & \text{ if $r$ is odd.}
\end{cases}
\end{align*}
See Figure~\ref{figure_example_of_a_matching}
for a matching on $\{1,2, \ldots, 9\}$
with $3$ crossings and $p=2$ as the isolated vertex.

\begin{itemize}
\item[$(A_{n})$]
The $2$-structures of a pseudo-root system of type $A_{n}$
have the type $A_{1}^{m}$ where $m = \lfloor(n+1)/2\rfloor$.
Here the matchings are on the set $\{1, 2, \ldots, n+1\}$,
that is, $r = n+1$.
The $2$-structure $\varphi$ indexed by the matching $M$ in
equation~\eqref{equation_matching}
is described by
\begin{align*}
\varphi
& =
\{\pm(e_{i_{1}}-e_{j_{1}})/\sqrt{2}, \pm(e_{i_{2}}-e_{j_{2}})/\sqrt{2}, \ldots, \pm(e_{i_{m}}-e_{j_{m}})/\sqrt{2}\} .
\end{align*}
Note that $\varphi$ has type $A_{1}^{m}$.

\item[$(D_{n})$]
The $2$-structures of a pseudo-root system of type $D_{n}$
have the type $A_{1}^{2m}$ where $m = \lfloor n/2\rfloor$.
Let $M$ be a matching on the vertex set $\{1,2, \ldots, n\}$, that is $r = n$,
given by equation~\eqref{equation_matching}.
Then the $2$-structure $\varphi$ indexed by the matching $M$ is
\begin{align*}
\varphi
& =
\{(\pm e_{i_{1}} \pm e_{j_{1}})/\sqrt{2}, (\pm e_{i_{2}} \pm e_{j_{2}})/\sqrt{2},
\ldots, (\pm e_{i_{m}} \pm e_{j_{m}})/\sqrt{2}\} .
\end{align*}
Here the $2$-structure $\varphi$ has type $A_{1}^{2m}$.
\end{itemize}
In both type $A_{n}$ and type $D_{n}$
the sign $\epsilon(\varphi)$ of the $2$-structure $\varphi$
is given by the sign of the matching~$M$, that is, $\epsilon(\varphi) = (-1)^{M}$.
This fact is straightforward to conclude using
Lemma~\ref{lemma_sign_for_2-structures}.

\section{Multivariate Taylor expansion of the pizza quantity}

We begin to consider the pizza quantity of an arrangement
of type $A_{1}^{k}$ in an $n$-dimensional space.
Let $\beta_{m}$ be the volume of the $m$-dimensional unit ball,
that is,
\begin{align*}
\beta_{m}
& =
\frac{\pi^{m/2}}{\Gamma(m/2 + 1)} .
\end{align*}

\begin{lemma}
Let $V$ be an $n$-dimensional space and
let $\Hf$ be the Coxeter arrangement of type $A_{1}^{k}$.
That is, $\Hf = \{H_{f}\}_{f \in E}$
where $E = \{f_{1}, \ldots, f_{k}\}$ is a set of $k$ orthogonal unit vectors.
The base chamber $T_{0}$ is the cone generated by the set $E$ and all vectors
orthogonal to the set~$E$.
Let $a$ be a point in $V$ such that $0 \in \Ball(a,1)$.
Then the pizza quantity is given by the $k$-dimensional integral
\begin{align*}
P(\Hf, \Ball(a, 1))
& =
2^{k} \cdot
\beta_{n-k} \cdot
\int_{0}^{(f_{1},a)}
\cdots
\int_{0}^{(f_{k},a)}
\left(1 - t_{1}^{2} - \cdots - t_{k}^{2}\right)^{(n-k)/2} \: dt_{1} \cdots dt_{k} .
\end{align*}
If $n$ and $k$ have the same parity,
then 
the pizza quantity is a polynomial in $a$ of degree at most~$n$.
\end{lemma}
\begin{proof}
Assume that the point $a$ lies in the base chamber $T_{0}$, that is, $(f_{i},a) \geq 0$ for all $i$.
The arrangement $\Hf$ consists of the $k$ hyperplanes
$H_{i} = \{x \in V : (f_{i},x) = 0\}$, $1 \leq i \leq k$,
and cuts the ball $\mathbb{B}(a,1)$ into $2^{k}$ pieces.
Furthermore, cut the ball with the affine hyperplanes
$\{x \in V : (f_{i},x) = 2 \cdot (f_{i},a)\}$,
yielding $3^{k}$ pieces.
We index these pieces by a sign vector $(c_{1}, \ldots, c_{k}) \in \{0, \pm 1\}^{k}$.
Let
\begin{align*}
H^{+1}_{i} & = \{x \in V : 2 \cdot (f_{i},a) \leq (f_{i},x)\} , \\
H^{0}_{i} & = \{x \in V : 0 \leq (f_{i},x) \leq 2 \cdot (f_{i},a)\} , \\
H^{-1}_{i} & = \{x \in V : (f_{i},x) \leq 0\} ,
\end{align*}
and finally let
$P(c_{1}, \ldots, c_{k}) = \mathbb{B}(a,1) \cap H^{c_{1}}_{1} \cap \cdots \cap H^{c_{k}}_{k}$.
Note that the two pieces
\begin{align*} 
P(0, \ldots, 0, +1, c_{i+1}, \ldots, c_{k})
\text{ and }
P(0, \ldots, 0, -1, c_{i+1}, \ldots, c_{k})
\end{align*}
are reflections of each other in the affine hyperplane $\{x \in V : (f_{i},x) = (f_{i},a)\}$
and hence they have the same volume.
But their volumes cancel in the pizza quantity.
By this sign-reversing involution, the only remaining piece is
$P(0, \ldots, 0)$,
hence
$P(\Hf, \Ball(a, 1)) = \Vol(P(0, \ldots, 0))$.
Pick $t_{1}$ through $t_{k}$ such that $-(f_{i},a) \leq t_{i} \leq (f_{i},a)$.
These values describe a point
\begin{align*}
x
& =
a
+
t_{1} \cdot f_{1}
+ \cdots +
t_{k} \cdot f_{k} ,
\end{align*}
that lies in the intersection $P(0,\ldots,0) \cap (\Span(E) + a)$.
The cross section 
of $P(0, \ldots, 0)$
parallel to~$E^{\perp}$ containing the point~$x$
is an $(n-k)$-dimensional ball of radius
$\sqrt{1 - t_{1}^{2} - \cdots - t_{k}^{2}}$.
Hence we conclude
\begin{align*}
P(\Hf, \Ball(a, 1))
& =
\beta_{n-k} \cdot
\int_{-(f_{1},a)}^{(f_{1},a)}
\cdots
\int_{-(f_{k},a)}^{(f_{k},a)}
\left(1 - t_{1}^{2} - \cdots - t_{k}^{2}\right)^{(n-k)/2} \: dt_{1} \cdots dt_{k} \\
& =
2^{k} \cdot \beta_{n-k} \cdot
\int_{0}^{(f_{1},a)}
\cdots
\int_{0}^{(f_{k},a)}
\left(1 - t_{1}^{2} - \cdots - t_{k}^{2}\right)^{(n-k)/2} \: dt_{1} \cdots dt_{k} .
\end{align*}
Lifting the restriction that the point $a$ lies in the base chamber $T_{0}$,
the above argument holds with volume replaced with signed volume.

Note that the integrand is a polynomial when $n-k$ is even.
In this case, the pizza quantity is a polynomial in $a$ of degree at most
$2 \cdot (n-k)/2 + k = n$.
\end{proof}

We now briefly consider the case
where the $2$-structures of the Coxeter arrangement have type~$A_{1}^{k}$
and $k$ has the same parity as the dimension $n$ of the space $V$.
\begin{proposition}
Let $V$ be an $n$-dimensional space. Let $\Phi$ be a pseudo-root system in $V$
of type different than $A_{1}^{n}$.
Let $a$ be a point such that $\|a\| \leq 1$.
Assume that the type of the $2$-structures of~$\Phi$ are $A_{1}^{k}$.
If $n$ and $k$ have the same parity then the pizza quantity
$P(\Hf_{\Phi}, \Ball(a, 1))$ vanishes.
\label{proposition_n_k_same_parity}
\end{proposition}
\begin{proof}
Let the valuation $v$ be given by $v(S) = \Vol(\Ball(a, 1) \cap S)$.
By Theorem~\ref{theorem_2-structure}
we express the pizza quantity as
the sum
\begin{align*}
P(\Hf, \Ball(a, 1))
& =
\sum_{\varphi \in \Tt(\Phi)} \epsilon(\varphi) \cdot P(\Hf_{\varphi}, \Ball(a, 1)) ,
\end{align*}
where $\varphi$ ranges over all $2$-structures of the 
pseudo-root system~$\Phi$.
Since $n$ and $k$ have the same parity,
each summand is a polynomial of degree at most $n$;
hence the pizza quantity is a polynomial of degree at most $n$
in the variable $a$.
Note that the zeros of this polynomial include the
hyperplanes of the arrangement $\Hf$, and there
are more than $n$ hyperplanes. Hence the polynomial is the zero polynomial.
\end{proof}

Note that this result is not as strong as Theorem~1.1
in~\cite{EMR_pizza}.
However, this result lets us concentrate on the case when
the parities of $n$ and $k$ differ.
Through the remainder of the paper we have this assumption.

Let $c_{m}$ be the coefficient
$(-1)^{m} \cdot \binom{(n-k)/2}{m}$
where the binomial coefficient is given by
$\binom{\alpha}{m} = \alpha \cdot (\alpha-1) \cdots (\alpha-m+1)/m!$.
Furthermore
let $T_{2m+k}$ be the homogenous polynomial
\begin{align*}
T_{2m+k}(x_{1}, \ldots, x_{k})
& = 
c_{m} \cdot
\sum_{r_{1} + \cdots + r_{k} = m}
\binom{m}{r_{1}, \ldots, r_{k}}
\cdot
\frac{x_{1}^{2r_{1}+1}}{2r_{1}+1} \cdots \frac{x_{k}^{2r_{k}+1}}{2r_{k}+1} .
\end{align*}
Furthermore, for $d < k$ let $T_{d} = 0$ and
for $d$ of opposite parity than $k$ also let $T_{d} = 0$.
Note that the polynomial $T_{d}$ is symmetric in the variables $x_{1}$ through $x_{k}$.
Moreover, $T_{d}$ is an odd function in each variable.

\begin{lemma}
Given the arrangement $\Hf = (H_{e_{i}})_{1 \leq i \leq k}$ of type $A_{1}^{k}$
and a point $x = (x_{1}, \ldots, x_{n})$ such that $\|x\| \leq 1$.
Then the pizza quantity $P(\Hf, \Ball(x, 1))$
has the multivariate Taylor series expansion
\begin{align*}
P(\Hf, \Ball(x, 1))
& =
2^{k} \cdot \beta_{n-k} \cdot \sum_{d \geq k} T_{d}(x_{1}, \ldots, x_{k})
\end{align*}
and this series converges for all $\|x\| \leq 1$.
\end{lemma}
\begin{proof}
Note that Taylor's theorem implies
\begin{align*}
(1-x)^{(n-k)/2} 
& =
\sum_{m \geq 0} c_{m} \cdot x^{m} ,
\end{align*}
which converges for $-1 \leq x \leq 1$.
Letting $x$ be the sum
$t_{1}^{2} + \cdots + t_{k}^{2}$ 
in this series yields
\begin{align*}
(1 - t_{1}^{2} - \cdots - t_{k}^{2})^{(n-k)/2}
& =
\sum_{m \geq 0}  c_{m} \cdot (t_{1}^{2} + \cdots + t_{k}^{2})^{m}  \\
& =
\sum_{m \geq 0}
c_{m} \cdot
\sum_{r_{1} + \cdots + r_{k} = m}
\binom{m}{r_{1}, \ldots, r_{k}}
\cdot
t_{1}^{2r_{1}} \cdots t_{k}^{2r_{k}} .
\end{align*}
Note that this power series converges for
$t_{1}^{2} + \cdots + t_{k}^{2} \leq 1$.
Now the result follows by integrating
each variable $t_{i}$ from $0$ to $x_{i}$.
Furthermore, it converges when $x_{1}^{2} + \cdots + x_{k}^{2} \leq 1$
which contains the ball $\|x\| \leq 1$.
\end{proof}

For $\psi$ a set of $k$ orthogonal unit vectors in the space $V$,
that is, $\psi = \{f_{1}, \ldots, f_{k}\}$, define
$T_{d}(\psi;a)$ to be the polynomial
\begin{align*}
T_{d}(\psi;a)
& =
T_{d}((f_{1},a), \ldots, (f_{k},a)) .
\end{align*}
Note that $T_{d}(\psi;a)$ is a polynomial in $a$ of degree $d$.
Furthermore, the polynomial is skew-symmetric for $s_{\alpha}$ where
$\alpha \in \psi$, that is,
$T_{d}(s_{\alpha}(\psi);a) = - T_{d}(\psi;a) = T_{d}(\psi;s_{\alpha}(a))$.
Lastly, it also satisfies
$T_{d}(w(\psi);w(a)) =  T_{d}(\psi;a)$ for all $w \in W$.

For a root system $\Phi$ whose $2$-structures have the type $A_{1}^{k}$,
define $Z_{d}(\Phi; a)$ to be the polynomial
\begin{align*}
Z_{d}(\Phi; a)
& =
\sum_{\varphi \in \Tt(\Phi)} \epsilon(\varphi) \cdot T_{d}(\varphi^{+}; a) .
\end{align*}

\begin{lemma}
The polynomial
$Z_{d}(\Phi; a)$
is skew-symmetric in the variable $a$ with respect to the action of the Coxeter group~$W$,
that is, the following identity holds
\begin{align*}
Z_{d}(\Phi; w(a))
& =
(-1)^{w} \cdot
Z_{d}(\Phi; a)
\end{align*}
for all $w \in W$.
Hence the Jacobian $J(a)$ divides the polynomial $Z_{d}(\Phi; a)$.
\label{lemma_Z_d_skew}
\end{lemma}
\begin{proof}
Since the Coxeter group $W$ acts on the set of $2$-structures $\Tt(\Phi)$
we can rewrite the sum as
\begin{align*}
Z_{d}(\Phi; w(a))
& =
\sum_{w(\varphi) \in \Tt(\Phi)} \epsilon(w(\varphi)) \cdot T_{d}((w(\varphi))^{+}; w(a)) .
\end{align*}
Note that $(w(\varphi))^{+} = w(\varphi) \cap \Phi^{+}$ consists of positive pseudo-roots.
However there exists an element~$v$ in the Coxeter group $W$
of the $2$-system $w(\varphi)$ such that
$v(w(\varphi^{+})) = (w(\varphi))^{+}$.
That is, $v$ is the commutative product of the reflections $s_{\alpha}$
where the pseudo-roots $\alpha$ are the negative roots in $w(\varphi^{+})$,
that is, $\alpha \in w(\varphi^{+}) \cap \Phi^{-} = w(\varphi^{+}) \cap (\Phi - \Phi^{+})$.
Using that the polynomial $T_{d}$ is skew-symmetric, we have that
$T_{d}(v(w(\varphi^{+})); w(a)) = (-1)^{v} \cdot T_{d}(w(\varphi^{+}); w(a))$.
Hence the above sum is given by
\begin{align*}
\sum_{w(\varphi) \in \Tt(\Phi)} \epsilon(w(\varphi)) \cdot T_{d}(v(w(\varphi^{+})); w(a))
& =
\sum_{w(\varphi) \in \Tt(\Phi)} \epsilon(w(\varphi)) \cdot (-1)^{v} \cdot T_{d}(w(\varphi^{+}); w(a)) \\
& =
\sum_{w(\varphi) \in \Tt(\Phi)} \epsilon(w(\varphi)) \cdot (-1)^{v} \cdot T_{d}(\varphi^{+}; a) .
\end{align*}
But note that $w(\varphi) = v(w(\varphi))$.
Since 
$v(w(\varphi^{+})) \subseteq \Phi^{+}$, Lemma~\ref{lemma_sign_for_2-structures} part~(a)
implies that the signs satisfy
$\epsilon(w(\varphi)) = \epsilon(v(w(\varphi))) = (-1)^{vw} \cdot \epsilon(\varphi) 
= (-1)^{w} \cdot \epsilon(\varphi) \cdot (-1)^{v}$.
Hence the sum is given by
\begin{align*}
(-1)^{w} \cdot \sum_{w(\varphi) \in \Tt(\Phi)} \epsilon(\varphi) \cdot T_{d}(\varphi^{+}; a) .
\end{align*}
Finally by replacing $w(\varphi) \in \Tt(\Phi)$ with $\varphi \in \Tt(\Phi)$ the result follows.
\end{proof}

\begin{theorem}
Let $\Hf$ be a Coxeter arrangement in a space $V$ of dimension $n$
with positive pseudo-roots $\Phi^{+}$.
Assume furthermore that the type of $2$-structures of $\Hf$ is $A_{1}^{k}$
where $n$ and $k$ have opposite parity.
If the ball $\Ball(a,1)$ contains the origin, the following multivariate Taylor expansion holds
\begin{align*}
P(\Hf, \Ball(a, 1))
& =
2^{k} \cdot \beta_{n-k} \cdot \sum_{d \geq |\Phi^{+}|} Z_{d}(\Phi; a) ,
\end{align*}
that is, there are no terms of degree less than $|\Phi^{+}|$.
\label{theorem_Taylor_expansion}
\end{theorem}
\begin{proof}
Using the valuation $v(S) = \Vol(\Ball(a,1) \cap S)$,
Theorem~\ref{theorem_2-structure} implies
\begin{align*}
P(\Hf, \Ball(a, 1))
& =
\sum_{\varphi \in \Tt(\Phi)} \epsilon(\varphi) \cdot P(\Hf_{\varphi}, \Ball(a, 1)) \\
& =
2^{k} \cdot \beta_{n-k} \cdot
\sum_{\varphi \in \Tt(\Phi)} \epsilon(\varphi) \cdot \sum_{d \geq 0} T_{d}(\varphi^{+}; a) \\
& =
2^{k} \cdot \beta_{n-k} \cdot \sum_{d \geq 0}
\sum_{\varphi \in \Tt(\Phi)} \epsilon(\varphi) \cdot T_{d}(\varphi^{+}; a) \\
& =
2^{k} \cdot \beta_{n-k} \cdot \sum_{d \geq 0} Z_{d}(\Phi;a) .
\end{align*}
By Lemma~\ref{lemma_Z_d_skew}
the degree~$d$ polynomial $Z_{d}(\Phi;a)$ is divisible by the Jacobian~$J(a)$.
But the Jacobian~$J(a)$ is a polynomial of degree~$|\Phi^{+}|$.
Thus there are no terms of degree less than~$|\Phi^{+}|$ in the multivariate Taylor expansion.
\end{proof}

\begin{table}
$$
\begin{array}{| c | c | c | c | c | c |}
\hline
&&&&& \\
\text{type} &n&k& |\Phi^{+}|  & Z_{|\Phi^{+}|}(\Phi; a)/J(a) & Z_{|\Phi^{+}|+2}(\Phi; a)/J(a)
 \\[4mm] \hline \hline
&&&&& \\
A_{2} &2&1& 3 &
\displaystyle \frac{1}{2}  
& \displaystyle \frac{3}{2^{5}} \cdot (a_{1}^{2} + a_{2}^{2} + a_{3}^{2})
\\[4mm] \hline
&&&&& \\
A_{3} &3&2& 6 &
\displaystyle -\frac{1}{2 \cdot 3}  
& \displaystyle -\frac{1}{2 \cdot 5} \cdot (a_{1}^{2} + \cdots + a_{4}^{2})
\\[4mm] \hline
&&&&& \\
A_{6} &6&3& 21 &
\displaystyle -\frac{3 \cdot 7 \cdot 11 \cdot 13}{2^{8}}
& \displaystyle \frac{3^{2} \cdot 5^{2} \cdot 7 \cdot 11 \cdot 13}{2^{13}} 
\cdot (a_{1}^{2} + \cdots + a_{7}^{2}) \\[4mm] \hline
&&&&& \\
A_{7} &7&4& 28 &
\displaystyle \frac{3 \cdot 11 \cdot 13 \cdot 17 \cdot 19}{2^{8}} 
& \displaystyle \displaystyle \frac{7 \cdot 11 \cdot 13 \cdot 17 \cdot 19}{2^{7}} 
\cdot (a_{1}^{2} + \cdots + a_{8}^{2})
\\[4mm] \hline
\end{array}
$$
\caption{The two leading terms of the multivariate Taylor series expansion of
$2^{-k} \cdot \beta_{n-k}^{-1} \cdot P(\Hf,\mathbb{B}(a,1))$
where the arrangement $\Hf$ has types $A_{2}$, $A_{3}$, $A_{6}$ and $A_{7}$.
Note that the relation $a_{1} + \cdots + a_{n+1} = 0$ holds.
In these types $k = \lfloor (n+1)/2 \rfloor$
and the number of positive pseudo-roots is given by $|\Phi^{+}| = \binom{n+1}{2}$.}
\label{table_A}
\end{table}

\section{Computational results for the leading term}

One consequence of 
Theorem~\ref{theorem_Taylor_expansion}
is that the leading term (if non-zero) has degree $|\Phi^{+}|$.
Hence if Conjecture~\ref{conjecture_conjecture} is true,
to get an idea of the correct sign, it is enough to calculate
this degree $|\Phi^{+}|$ term.
Tables~\ref{table_A} and~\ref{table_D} show the results
of these calculations in types $A_{n}$ and $D_{n}$
up to dimension $n=7$.

The next case to compute in type~$A$ is dimension~$10$.
In this case we have $11!! = 11 \cdot 9 \cdot 7 \cdot 5 \cdot 3 \cdot 1 = 10395$
number of matchings on the set $\{1,2, \ldots, 11\}$.
Each matching is a polynomial that expands into $2^{5} = 32$ terms.
In total, the calculation involves $2^{5} \cdot 11!! = 332640$ monomials.
In type~$D$, the next case is in dimension~$9$
and the similar calculation yields $2^{4} \cdot 9!! = 15120$ monomials.

It is interesting to note that the coefficients
$Z_{|\Phi^{+}|}(\Phi; a)/J(a)$ in all these cases are
rational numbers that factor nicely into small primes,
suggesting that there is an explicit expression.

Based on this data we can now refine Conjecture~\ref{conjecture_conjecture} as follows.
\begin{conjecture}
Let $\Hf$ be a Coxeter arrangement of type $A_{n}$ or $D_{n}$
in an $n$-dimensional space~$V$.
Let $a$ be a point in $V$ such that $0 \in \Ball(a,R)$
and assume that the point $a$ lies in the interior of a chamber $T$ of the arrangement $\Hf$.
\begin{itemize}
\item[$(A_{n})$]
Assuming that $n \equiv 2$ or $3 \bmod 4$ then
the pizza quantity~$P(\Hf,\Ball(a,R))$ has its sign given by
\begin{align*}
(-1)^{\lfloor (n+1)/4 \rfloor} \cdot (-1)^{T} \cdot P(\Hf,\Ball(a,R)) > 0 .
\end{align*}
\item[$(D_{n})$]
Assuming that $n$ is odd then
the pizza quantity~$P(\Hf,\Ball(a,R))$ has its sign given by
\begin{align*}
(-1)^{T} \cdot P(\Hf,\Ball(a,R)) < 0 .
\end{align*}
\end{itemize}
\label{conjecture_A_and_D}
\end{conjecture}

Observe that the calculations in Tables~\ref{table_A} and~\ref{table_D}
imply that for dimension at most $7$
there is a neighborhood $U$ of the origin
such that the conjecture is true for $a$ in this neighborhood $U$.

\begin{table}
$$
\begin{array}{| c | c | c | c | c | c |}
\hline
&&&&& \\
\text{type} &n&k& |\Phi^{+}|  & Z_{|\Phi^{+}|}(\Phi; a)/J(a) & Z_{|\Phi^{+}|+2}(\Phi; a)/J(a)
 \\[4mm] \hline \hline
&&&&& \\
D_{3} &3&2& 6 &
\displaystyle -\frac{1}{2 \cdot 3}
& \displaystyle -\frac{1}{2 \cdot 5} \cdot (a_{1}^2 + a_{2}^{2} + a_{3}^{2})
\\[4mm] \hline
&&&&& \\
D_{5} &5&4& 20 &
\displaystyle - \frac{11 \cdot 13}{2^{3} \cdot 5}  
& \displaystyle - \frac{11 \cdot 13}{2^{2} \cdot 3}  \cdot (a_{1}^2 + \cdots + a_{5}^{2})
\\[4mm] \hline
&&&&& \\
D_{7} &7&6& 42 &
\displaystyle - \frac{11 \cdot 13 \cdot 17 \cdot 19 \cdot 23 \cdot 29 \cdot 31}{2^{4} \cdot 3 \cdot 7}
& \displaystyle - \frac{5 \cdot 11 \cdot 17 \cdot 19 \cdot 23 \cdot 29 \cdot 31}{2^{4}}
\cdot (a_{1}^2 + \cdots + a_{7}^{2})
\\[4mm] \hline
\end{array}
$$
\caption{The two leading terms of the multivariate Taylor series expansion of $2^{-n} \cdot P(\Hf,\mathbb{B}(a,1))$
where the arrangement $\Hf$ has types $D_{3}$, $D_{5}$ and $D_{7}$.
Here $k = n-1$
and $|\Phi^{+}| = n \cdot (n-1)$.
Note that the calculation for type $D_{3}$ agrees with type $A_{3}$ in Table~\ref{table_A}.}
\label{table_D}
\end{table}

\section{Type $D_{n}$ in odd dimensions}

We now consider a pseudo-root system of type $D_{n}$
where $n$ is odd and $k = n-1$ is even.
Let $s_{i}$ denote the orthogonal reflection in the
coordinate hyperplane
$\{(x_{1}, \ldots, x_{n}) \in \Rrr^{n} : x_{i} = 0\}$.
We first explore the form of the polynomial $Z_{d}(\Phi; a)$.
\begin{lemma}
The polynomial $Z_{d}(\Phi; a)$, which is homogenous of degree $d$, satisfies the following properties:
\begin{itemize}
\item[(a)]
It is invariant under the reflection $s_{i}$, that is,
$Z_{d}(\Phi; s_{i}(a)) = Z_{d}(\Phi; a)$.
In other words, $Z_{d}(\Phi; a)$ is a polynomial in the variables
$a_{1}^{2}, \ldots, a_{n}^{2}$.
\item[(b)]
The polynomial $Z_{d}(\Phi; a)$ contains no monomial
where all the exponents are positive.
\item[(c)]
The polynomial $Z_{d}(\Phi; a)$ contains no monomial
where two exponents are equal.
\end{itemize}
\label{lemma_Z_d_basic_properties}
\end{lemma}
\begin{proof}
To prove part~(a) it is enough to verify that
$T_{d}(\varphi^{+}; s_{i}(a)) = T_{d}(\varphi^{+}; a)$
for a $2$-structure $\varphi$.
Let $M = \{(i_{1}<j_{1}), \ldots, (i_{m}<j_{m})\}$
be a maximal matching on the set $\{1, \ldots, n\}$.
The positive roots of the associated $2$-structure are given by
\begin{align*}
\varphi^{+}
& =
\left\{f^{+}_{1}, f^{-}_{1}, f^{+}_{2}, f^{-}_{2}, \ldots, f^{+}_{k/2}, f^{-}_{k/2}\right\} ,
\end{align*}
where $f^{+}_{r} = (e_{i_{r}} + e_{j_{r}})/\sqrt{2}$
and $f^{-}_{r} = (e_{i_{r}} - e_{j_{r}})/\sqrt{2}$.
If $i$ is the isolated vertex of matching $M$ then each vector of $\varphi^{+}$ is
invariant under the reflection $s_{i}$ and since $s_{i}$ is self-adjoint
we have
\begin{align*}
T_{d}(\varphi^{+}; s_{i}(a))
& =
T_{d}((f^{+}_{1},s_{i}(a)), \ldots, (f^{-}_{k/2},s_{i}(a)))
=
T_{d}((s_{i}(f^{+}_{1}),a), \ldots, (s_{i}(f^{-}_{k/2}),a)) \\
& =
T_{d}((f^{+}_{1},a), \ldots, (f^{-}_{k/2},a))
=
T_{d}(\varphi^{+}; a) .
\end{align*}
If $i$ is equal to $j_{r}$ for some index $r$ then
again the vectors of $\varphi^{+}$ are invariant under $s_{i}$
except for $s(f^{+}_{r}) = f^{-}_{r}$ and $s(f^{-}_{r}) = f^{+}_{r}$.
Since the polynomial $T_{d}$ is symmetric the calculation carries through.
Finally, if $i$ is equal to $i_{r}$ for some index $r$
then the only difference is that
for $s(f^{+}_{r}) = - f^{-}_{r}$ and $s(f^{-}_{r}) = - f^{+}_{r}$.
Since $T_{d}$ is odd in each variable, the two negative signs will cancel
in the calculation.

Let $p$ be the isolated vertex of a matching $M$.
Then for the associated $2$-structure $\varphi$
the expression
$T_{d}(\varphi^{+};a)$ is independent of the variable $a_{r}$.
From this observation part~(b) follows.

Part~(c) follows from Lemma~\ref{lemma_Z_d_skew}.
Since if $\cdots a_{i}^{e} \cdots a_{j}^{e} \cdots$
were such a monomial with non-zero coefficient $c$,
the reflection in the hyperplane $H_{(e_{i}-e_{j})/\sqrt{2}}$
would change the sign of this monomial.
This action contradicts that the coefficient~$c$ was non-zero.
\end{proof}

\begin{conjecture}
Let $Y$ be the polynomial
\begin{align}
Y
& =
\sum_{M} (-1)^{M} \cdot
T_{d}\bigl(\bigl\{f^{+}_{1}, f^{-}_{1}, \ldots, f^{+}_{k/2}, f^{-}_{k/2}\bigl\}, a\bigl)  ,
\label{equation_a_n_zero}
\end{align}
where $M = \{(i_{1}<j_{1}), (i_{2}<j_{2}), \ldots, (i_{k/2}<j_{k/2})\}$
ranges over all maximal matchings of the set $\{1, \ldots, n-1\}$
and $f^{\pm}_{r} = (e_{i_{r}} \pm e_{j_{r}})/\sqrt{2}$.
Consider a monomial $\mu$
where the exponents are decreasing (and hence strictly decreasing),
that is, it has the form
\begin{align*}
\mu
& =
a_{1}^{2\lambda_{1}+2(n-1)} \cdot
a_{2}^{2\lambda_{2}+2(n-2)} \cdots
a_{n-1}^{2\lambda_{n-1}+2} ,
\end{align*}
where $\lambda = (\lambda_{1}, \ldots, \lambda_{n-1})$ is a partition of $(d - n \cdot (n-1))/2$
into $n-1$ parts.
Then the coefficient of $\mu$ in the polynomial $Y$ is negative.
\label{conjecture_D_by_degree} 
\end{conjecture}

\begin{proposition}
Conjecture~\ref{conjecture_D_by_degree} 
implies 
part $(D_{n})$ of Conjecture~\ref{conjecture_A_and_D}.
\end{proposition}
\begin{proof}
Begin by observing that setting the last variable $a_{n}$ to zero
in $Z_{d}(\Phi; a)$ yields the polynomial~$Y$.
Next, any monomial in $Z_{d}(\Phi; a)$
has the form 
\begin{align}
a_{\sigma(1)}^{2\lambda_{1}+2(n-1)} \cdot
a_{\sigma(2)}^{2\lambda_{2}+2(n-2)} \cdots
a_{\sigma(n-1)}^{2\lambda_{n-1}+2} ,
\label{equation_lambda_monomial}
\end{align}
for some permutation $\sigma$ in the symmetric group $\mathfrak{S}_{n}$.
Since the symmetric group $\mathfrak{S}_{n}$ is a subgroup of the Coxeter group $W$,
and the action of $\mathfrak{S}_{n}$ is skew-symmetric,
we hence know that all the monomials are of the form
\begin{align*}
(-1)^{\sigma}
\cdot
a_{\sigma(1)}^{2\lambda_{1}+2(n-1)} \cdot
a_{\sigma(2)}^{2\lambda_{2}+2(n-2)} \cdots
a_{\sigma(n-1)}^{2\lambda_{n-1}+2} .
\end{align*}
Now summing over all permutations $\sigma$ in $\mathfrak{S}_{n}$ yields the determinant
\begin{align}
\label{equation_sum}
&
\sum_{\sigma \in \mathfrak{S}_{n}}
(-1)^{\sigma}
\cdot
a_{\sigma(1)}^{2\lambda_{1}+2(n-1)} \cdot
a_{\sigma(2)}^{2\lambda_{2}+2(n-2)} \cdots
a_{\sigma(n-1)}^{2\lambda_{n-1}+2}  \\
= &
\nonumber
\det
\begin{pmatrix}
a_{1}^{2\lambda_{1}+2(n-1)} & a_{2}^{2\lambda_{1}+2(n-1)} & \cdots & a_{n}^{2\lambda_{1}+2(n-1)} \\
a_{1}^{2\lambda_{2}+2(n-2)} & a_{2}^{2\lambda_{2}+2(n-2)} & \cdots & a_{n}^{2\lambda_{2}+2(n-2)} \\
\vdots & \vdots & \ddots & \vdots \\
a_{1}^{2\lambda_{n-1}+2} & a_{2}^{2\lambda_{n-1}+2} & \cdots & a_{n}^{2\lambda_{n-1}+2} \\
1 & 1 & \cdots & 1
\end{pmatrix} .
\end{align}
But
the bialternant formula of Jacobi
factors this determinant as the Vandermonde product times
the Schur function~$s_{\lambda}$, both in the variables $a_{1}^{2}, \ldots, a_{n}^{2}$.
Thus the sum in~\eqref{equation_sum} is given by
\begin{align*}
2^{n \cdot (n-1)/2} \cdot J(a) \cdot s_{\lambda}(a_{1}^{2}, a_{2}^{2}, \ldots, a_{n}^{2}) .
\end{align*}
By the combinatorial interpretation of the Schur function,
we know that
$s_{\lambda}(a_{1}^{2}, a_{2}^{2}, \ldots, a_{n}^{2})$
will always be non-negative.
Since we assumed that the monomial in equation~\eqref{equation_lambda_monomial}
when $\sigma$ is the identity permutation has a negative coefficient, we know
that if we sum the monomials in~\eqref{equation_sum} with their coefficients
we obtain that $Z_{d}(\Phi;a)$ is negative.
But this is the degree $d$ part of the multivariate Taylor expansion of the pizza quantity,
implying part $(D_{n})$ of the conjecture.
\end{proof}

For the case $n=3$ we have the following strengthening of
Conjecture~\ref{conjecture_D_by_degree}.
\begin{conjecture}
For $d \geq 6$, the polynomial
$-T_{d}\left((a_{1}+a_{2})/{\sqrt{2}}, (a_{1}-a_{2})/{\sqrt{2}}\right)/
(a_{1}^{2} - a_{2}^{2})$
belongs to
$\Rrr_{\geq 0}[a_{1}^{2}+a_{2}^{2}, a_{1}^{2} \cdot a_{2}^{2}]$.
\end{conjecture}

\section{Concluding remarks}

The other pizza theorem proved in~\cite{EMR_pizza}
states that for a Coxeter arrangement $\Hf$, different from~$A_{1}^{n}$, such that
the negative of the identity map belongs to the associated Coxeter group $W$,
and for a measurable set $K$, stable under the action of $W$, such that
the convex hull of $\{w(a) : w \in W\}$ is contained in $K$, then
the pizza quantity vanishes, that is,
$P(\Hf, K+a) = 0$;
see~\cite[Theorem~1.2]{EMR_pizza}.
For the case when the Coxeter group $W$ does not contain the negative map,
can we obtain a multivariate Taylor series for the pizza quantity?
In the irreducible types, these cases would be $A_{n}$ for $n \geq 2$,
$D_{n}$ for $n$ odd, the sporadic type~$E_{6}$
and the dihedral type $I_{2}(m)$ where $m$ is odd.

Note that Proposition~\ref{proposition_n_k_same_parity}
uses analysis in its proof. 
There are two pizza theorems in~\cite[Theorems~1.1 and~1.2]{EMR_pizza}.
The first result is Theorem~\ref{theorem_second_in_introduction}
and the second one is discussed in the previous paragraph.
The paper~\cite{EMR_pizza_2} provides a dissection proof
of the second theorem.
Hence it would be interesting if
the usage of analysis could be removed from
the proof of Proposition~\ref{proposition_n_k_same_parity}
such that we have a completely non-analytic proof of the result.

\begin{question}
{\rm
Can the following linear map be described in a more concise way?
Given a root system $\Phi$ whose $2$-structures have type $A_{1}^{k}$
and a symmetric polynomial $p(x_{1}, \ldots, x_{k})$ in
$k$ variables. Note that the polynomial $p(x_{1}^{2}, \ldots, x_{k}^{2})$
is stable under the action of type~$A_{1}^{k}$,
that is,
$p(x_{1}^{2}, \ldots, x_{k}^{2})$
is invariant under the substitution $x_{i} \longmapsto -x_{i}$.
Let $q(x_{1}, \ldots, x_{k}) = x_{1} \cdots x_{k} \cdot p(x_{1}^{2}, \ldots, x_{k}^{2})$,
that is, this polynomial is skew-symmetric with respect to $A_{1}^{k}$.
For a $2$-structure $\varphi$ of~$\Phi$ let
\begin{align*}
q(\varphi; a) & = q((e_{1},a), \ldots, (e_{k},a)) ,
\end{align*}
where $\varphi^{+} = \{e_{1}, \ldots, e_{k}\}$.
Finally, define $r(a)$ to be the sum
\begin{align*}
r(a)
& =
\sum_{\varphi}
\epsilon(\varphi) \cdot q(\varphi; a) .
\end{align*}
Observe that the polynomial $r(a)$ is divisible by the Jacobian of $\Phi$.
Note that the map sending $p$ to the quotient $r/J$ is linear.
Hence, is there a better way to understand 
the linear map $p \longmapsto r/J$?
}
\end{question}

\section*{Acknowledgements}

The author thanks Theodore Ehrenborg
and the referees for their comments
on an earlier version of this paper.
This work was partially supported by a grant from the
Simons Foundation (\#854548 to Richard Ehrenborg).

\bibliographystyle{amsplain}
\bibliography{bib}

\end{document}